\documentclass[12pt]{amsart}
\usepackage{amssymb}
\usepackage{amsmath}
\usepackage{amsthm}

\newtheorem {theorem} {Theorem}[section]
\newtheorem {proposition} [theorem]{Proposition}
\newtheorem {corollary} [theorem]{Corollary}
\newtheorem {lemma}  [theorem]{Lemma}
\numberwithin{equation}{section}
\theoremstyle{definition}
\newtheorem{df}{Definition}

\newtheorem{prb}{Problem}
\theoremstyle{remark}
\newtheorem{rem}{Remark}

\newcommand{\al}{\alpha}

\newcommand{\clk}{\sigma}
\newcommand{\bclk}{\mathfrak{S}}
\newcommand{\neva}{N}
\newcommand{\nevva}{\mathcal{N}}
\newcommand{\bneva}{\mathfrak{N}}
\newcommand{\aream}{A}
\newcommand{\si}{\Sigma}
\newcommand{\plh}{PM}
\newcommand{\dcm}{DM}
\newcommand{\lad}{\lambda}
\newcommand{\ant}{a}
\newcommand{\diin}{\theta}
\newcommand{\Mfld}{\mathcal{M}}
\newcommand{\cau}{\mathcal{K}}

\newcommand{\er}{\varepsilon}
\newcommand{\za}{\zeta}

\newcommand{\mPhi}{\Sigma_{\Phi}}
\newcommand{\ph}{\varphi}
\newcommand{\cph}{C_\varphi}
\newcommand{\hol}{\mathcal{H}ol}
\newcommand{\Dbb}{\mathbb D}
\newcommand{\Tbb}{\mathbb T}
\newcommand{\Rl}{\mathrm{Re}}

\newcommand{\sptwo}{S_2}
\newcommand{\spd}{S_d}
\newcommand{\wsd}{\widehat\si_d}
\newcommand{\cd}{{\mathbb{C}}^d}
\newcommand{\ppd}{{\mathbb{P}}^{d-1}}
\newcommand{\bd}{B_d}

\newcommand{\Cbb}{\mathbb C}
\newcommand{\Nbb}{\mathbb N}
\newcommand{\kla}{I^*(H^2)}
\newcommand{\ksm}{I_*(H^2)}

\begin{document}

\title[Clark measures]{Clark measures on the complex sphere}



\author{Aleksei B.\ Aleksandrov}
\address{St.~Petersburg Department
of Steklov Mathematical Institute, Fontanka 27, St.~Petersburg
191023, Russia}

\email{alex@pdmi.ras.ru}

\author{Evgueni Doubtsov}

\address{St.~Petersburg Department
of Steklov Mathematical Institute, Fontanka 27, St.~Petersburg
191023, Russia}

\email{dubtsov@pdmi.ras.ru}

\thanks{This research was supported by the Russian Science Foundation (grant No.~18-11-00053).}

\subjclass[2010]{Primary 32A35; Secondary 30E20, 30J05, 31C10, 32A26, 32A40, 43A85, 46E27, 46J15, 47A45, 47B33}

\keywords{Hardy space, inner function, model space, pluriharmonic measure, representing measure,
Henkin measure, Cauchy integral,
composition operator}

\begin{abstract}
Let $B_d$ denote the unit ball of $\mathbb{C}^d$, $d\ge 1$.
Given a holomorphic function $\varphi: B_d \to B_1$, we study the corresponding family
$\sigma_\alpha[\varphi]$, $\alpha\in\partial B_1$,
of Clark measures
on the unit sphere $\partial B_d$.
If $\varphi$ is an inner function, then we introduce and investigate related unitary operators $U_\alpha$ mapping analogs of model spaces onto $L^2(\sigma_\alpha)$, $\alpha\in\partial B_1$.
In particular, we explicitly characterize the set of $U_\alpha^* f$ such that $f\sigma_\alpha$
is a pluriharmonic measure.
Also, for an arbitrary holomorphic $\varphi: B_d \to B_1$, we use the family $\sigma_\alpha[\varphi]$ to
compute the essential norm of the composition operator $C_\varphi: H^2(B_1)\to H^2(B_d)$.
\end{abstract}

\maketitle

\section{Introduction}\label{s_int}

Let $\bd$ denote the open unit ball of $\cd$, $d\ge 1$. For the unit disk $B_1$ of $\Cbb$, we also use the notation $\Dbb$.
Put $\spd = \partial\bd$ and $\Tbb=\partial\Dbb$.
In fact, our notation is close to that in monograph \cite{Ru80}.
Also, we often use without explicit references
basic results of the function theory in $\bd$
presented in \cite{Ru80}.
So, for $z,\za\in \bd\cup \spd$ with $\langle z, \za \rangle \neq 1$, the equality
\[
C(z, \za) = \frac{1}{(1-\langle z, \za \rangle)^d}
\]
defines the Cauchy kernel for $\bd$.
The invariant Poisson kernel is given by the formula
\[
P(z, \za) = \frac{C(z, \za) C(\za, z)}{C(z, z)} = \left(\frac{1-|z|^2}{|1-\langle z, \za \rangle|^2} \right)^d,
\quad z,\za\in \bd\cup \spd\ \textrm{with}\  \langle z, \za \rangle \neq 1.
\]

\subsection{Pluriharmonic measures}
Let $M(\spd)$ denote the space of complex Borel measures on the sphere $\spd$.
A measure $\mu\in M(\spd)$ is called \textsl{pluriharmonic} if the invariant Poisson integral
\[
P[\mu](z) = \int_{\spd} P(z, \za)\, d\mu(\za), \quad z\in\bd,
\]
is a pluriharmonic function.
Let $\plh(\spd)$ denote the set of all pluriharmonic measures.
For $\mu\in\plh(\spd)$, it is well-known that the invariant Poisson integral $P[\mu]$
coincides with the harmonic one.

\subsection{Clark measures}
Given an $\alpha\in\Tbb$ and a holomorphic function $\ph: \bd\to \Dbb$, the quotient
\[
\frac{1-|\ph(z)|^2}{|\al-\ph(z)|^2}= \Rl \left(\frac{\al+ \ph(z)}{\al- \ph(z)} \right), \quad z\in \bd,
\]
is positive and pluriharmonic.
Therefore, there exists a unique positive measure $\clk_\al= \clk_\al[\ph] \in M(\spd)$
such that
\[
P[\clk_\al](z) = \Rl \left(\frac{\al+ \ph(z)}{\al- \ph(z)} \right), \quad z\in \bd.
\]
Clearly, $\clk_\al\in \plh(\spd)$ by the definition of $\plh(\spd)$.

After the famous paper of Clark \cite{Cl72},
various properties and applications of the measures $\clk_\al$ on the unit circle $\Tbb$ have been obtained;
see, for example, reviews \cite{MM06, PS06, Sa07, GR15} for further references.
To the best of the authors' knowledge, the measures $\clk_\al$ on the unit sphere $\spd$, $d\ge 2$, have not been
investigated earlier.
See \cite{Ju14} for a different extension of the Clark theory motivated by the multivariable operator theory.

\subsection{Clark measures and model spaces}
Let $\si=\si_d$ denote the normalized Lebesgue measure on the sphere $\spd$.

\begin{df}\label{d_inner}
A
holomorphic function $I:\bd \to \Dbb$ is called \textsl{inner}
if $|I(\za)|=1$ for $\si_d$-a.e. $\za\in\spd$.
\end{df}

In the above definition,
$I(\za)$ stands, as usual, for
$\lim_{r\to 1-} I(r\za)$.
Recall that the corresponding limit is known to exist $\si_d$-a.e.
Also, by the above definition, unimodular constants are not inner functions.

The present paper is primarily motivated by the studies of Clark \cite{Cl72}
related to the measures $\clk_\al [\ph] \in M(\Tbb)$, $\al\in\Tbb$,
where $\ph$ is an inner function in $\Dbb$.

Given an inner function $I$ is $\bd$, we have
\[
P[\clk_\al](\za) =\frac{1-|I(\za)|^2}{|\al-I(\za)|^2}=0 \quad \si_d\textrm{-a.e.},
\]
thus, $\clk_\al = \clk_\al[I]$ is a singular measure. Here and in what follows, this means that
$\clk_\al$ is singular with respect to $\si_d$; in brief, $\clk_\al \bot\si_d$.

For $d\ge 1$, let $\hol(\bd)$ denote the space of holomorphic functions in $\bd$.
For $0<p<\infty$, the classical Hardy space $H^p=H^p(\bd)$ consists of those $f\in \hol(\bd)$ for which
\[
\|f\|_{H^p}^p = \sup_{0<r<1} \int_{\spd} |f(r\za)|^p\, d\si_d(\za) < \infty.
\]
As usual, we identify the Hardy space $H^p(\bd)$, $p>0$, and the space
$H^p(\spd)$ of the corresponding boundary values.

Given an inner function $\diin$ on $\Dbb$, the classical
model space $K_\diin$ is defined as $K_\diin = H^2(\Dbb)\ominus \diin H^2(\Dbb)$.
Clark \cite{Cl72} introduced and studied a family of unitary operators
$U_\al : K_\diin \to L^2(\clk_\al)$, $\al\in\Tbb$.

For an inner function $I$ in $\bd$, $d\ge 2$, consider the following natural analogs of $K_\diin$:
\begin{align*}
  \kla &= H^2 \ominus I H^2;\\
  \ksm &= \{f\in H^2:\, I\overline{f} \in H^2_0\},
\end{align*}
where $H^2_0 = \{f\in H^2: f(0)=0\}$.
Clearly, we have $\ksm \subset \kla$;
if $\diin$ is an inner function in $\Dbb$, then $\diin^* (H^2(\Dbb)) = \diin_* (H^2(\Dbb)) = K_\diin$.
In this paper, we define unitary operators
\[
U_\al: \kla\to L^2(\clk_\al), \quad \al\in\Tbb,
\]
and we obtain the following characterization:

\begin{theorem}\label{t_ksm_plh}
Let $I$ be an inner function in the unit ball $\bd$, $d\ge 2$,
and let $f\in L^2(\clk_\al)$, $\al\in\Tbb$.
Then the following properties are equivalent:
\begin{itemize}
  \item [(i)] $U_\al^* f \in \ksm$;
  \item [(ii)] $f\clk_\al \in \plh(\spd)$.
\end{itemize}
\end{theorem}

\subsection{Clark measures and essential norms of composition operators}
A different application of the Clark measures on the unit sphere comes from the
studies of composition operators.
Namely, each holomorphic function $\ph: \bd\to \Dbb$, $d\ge 1$, generates the composition operator
$\cph: \hol(\Dbb) \to \hol(\bd)$ by the following formula:
\[
(\cph f)(z) = f(\ph(z)), \quad z\in\bd.
\]
It is well known that $\cph$ maps $H^2(\Dbb)$
into $H^2(\bd)$, $d\ge 1$.
So, a natural problem is to characterize the compact operators $\cph: H^2(\Dbb) \to H^2(\bd)$.
A more general problem is to compute or estimate the essential norm of the composition operator under consideration.
For $d=1$, a solution to this problem in terms of the Nevanlinna counting function was given
in the seminal paper of J.~H.~Shapiro \cite{ShJoel87}.
A solution in terms of the family $\sigma_\al[\varphi]$, $\al\in\Tbb$, was later
obtained by Cima and Matheson \cite{CM97}.
Extending the theorem of Cima and Matheson to all $d\ge 1$, we prove the following result:

\begin{theorem}\label{t_cmp_H2K_clk}
Let $\ph: \bd\to \Dbb$, $d\ge 1$, be a holomorphic function.
Then the essential norm of the composition operator $\cph: H^2(\Dbb) \to H^2(\bd)$ is equal to the following quantity:
\[
\sqrt{\sup \{ \|\clk_\al^s\|: \al\in\Tbb\} },
\]
where $\clk_\al^s$ denotes the singular part of the Clark measure $\clk_\al= \clk_\al[\ph]$.
\end{theorem}

\subsection*{Organization of the paper}
Properties of general pluriharmonic measures and Clark measures on the unit sphere
are studied in Sections~\ref{s_aux_plh} and \ref{s_aux_clk}, respectively.
On the one hand, we show that any pluriharmonic
measure is the integral, in an appropriate weak sense, of its canonical slice measures;
see Proposition~\ref{p_slices} and Theorem~\ref{t_DM}.
On the other hand, the normalized Lebesgue measure on the unit sphere $\spd$
is the integral, in a weak sense and with respect to $\al\in\Tbb$, of the measures $\clk_\al$;
see Theorem~\ref{t_desint}.
Also, Cauchy integrals of the Clark measures on the unit sphere are studied in Section~\ref{s_aux_clk}.
Clark measures generated by inner functions are considered in Section~\ref{s_clk_inner}.
On this way, we also give the negative answer to
a question asked by the first author \cite[Problem~1]{Aab86}; see Theorem~\ref{t_wlim_M0}.
Applications are collected in Sections~\ref{s_model} and \ref{s_clk_cmp}.
So, Theorem~\ref{t_ksm_plh} and other results related to the unitary operators $U_\al: \kla\to L^2(\clk_\al)$, $\al\in\Tbb$,
are obtained in Section~\ref{s_model}.
Relations between properties of composition operators and Clark measures on the unit sphere
are studied in the final Section~\ref{s_clk_cmp}; in particular, we prove Theorem~\ref{t_cmp_H2K_clk}.

The main results of Section~\ref{s_model} were announced in \cite{ADcras}.

\section{Pluriharmonic measures}\label{s_aux_plh}

\subsection{A decomposition theorem for pluriharmonic measures}
Let $\ppd$ denote
the complex projective space of dimension $d-1$,
that is, the collection of all one-dimensional
linear subspaces of $\cd$.
Let $\wsd$ denote the unique probability measure on $\ppd$
 invariant with respect to all unitary transformations of $\cd$.
For every Borel set $E\subset\ppd$, we clearly have
 $\wsd(E)=\si_d(\pi^{-1}(E))$, where
$\pi = \pi_d$ denotes the
canonical projection from $\spd$ onto $\ppd$.

It is well-known and easy to see that
\begin{equation}\label{e_Leb_dcm}
\int_{\spd} f\, d\si_d = \int_{\ppd} \int_{\spd} f \, d\si_1^\za\, d\wsd(\za)
\quad\textrm{for all\ } f\in C(\spd),
\end{equation}
where $\si_1^\za \in M(\spd)$ is the normalized one-dimensional Lebesgue measure
supported by the unit circle $\pi^{-1}(\za) \subset \spd$.

The above property of $\si_d$ leads to the following definition.

\begin{df}\label{d_decomp_M}
A measure $\mu\in M(\spd)$ is called \textsl{decomposable}
if, for $\wsd$-a.e.\ $\za\in\ppd$, there exists a measure $\mu_\za\in M(\spd)$ such that
$\textrm{supp\,}\mu_\za \subset \pi^{-1}(\za)$,
\begin{equation}\label{e_int_norms_muz}
\int_{\ppd} \|\mu_\za \|\, d\widehat\si_d (\za) < \infty
\end{equation}
and
\begin{equation}\label{int_decom}
\mu = \int_{\ppd} \mu_\za\, d\widehat\si_d (\za)
\end{equation}
in the following weak sense:
\begin{equation}\label{e_slice_sphsph_C_proj}
\int_{\spd} f\, d\mu = \int_{\ppd} \int_{\spd} f \, d\mu_\za\, d\widehat\si_d(\za)
\end{equation}
for all $\ f\in C(\spd)$.
Let $\dcm(\spd)$ denote the set of all decomposable measures $\mu\in M(\spd)$.
\end{df}

\begin{rem}\label{r_rokh}
There are definitions similar to the above notion; see, for example, \cite{Rh52}.
However, to the best of the authors' knowledge, Definition~\ref{d_decomp_M} is a more specific one.
\end{rem}

\begin{rem}\label{r_DM}
If $\mu\in\dcm(\spd)$, then
\eqref{e_slice_sphsph_C_proj} holds for
every bounded Borel function $f$ on $\spd$.
Indeed, if $f_n$, $n=1, 2, \dots$, are Borel functions on $\spd$,
$|f_n|\le C$ and \eqref{e_slice_sphsph_C_proj} holds for
$f_n$, then \eqref{e_slice_sphsph_C_proj} holds for
$f(\xi) =\lim_{n\to \infty} f_n(\xi)$, $\xi\in\spd$,
by \eqref{e_int_norms_muz} and the dominated convergence theorem.
So,
starting with the continuous functions on $\spd$ and arguing by transfinite induction,
we conclude that \eqref{e_slice_sphsph_C_proj} holds for all bounded functions in any Baire class,
hence, for all bounded Borel functions.
\end{rem}

In fact, Theorem~\ref{t_DM} below guarantees that \eqref{e_slice_sphsph_C_proj}
holds for any $f\in L^1(|\mu|)$.

Now, let $\mu$ be a pluriharmonic measure. Put $u=P[\mu]$
and $u_r(\xi) = u(r\xi)$, $0\le r <1$, $\xi\in\spd$.
For $\xi\in\spd$, the slice function $u_\xi$ is defined as $u_\xi(z) = u(z\xi)$, $z\in\Dbb$.
Since $u$ is pluriharmonic, $u_\xi$ is harmonic for all $\xi\in\spd$.
Also, by the monotone convergence theorem, we have
\begin{equation}\label{e_u_xi}
\begin{aligned}
\int_{\spd}\sup_{0<r<1}\|(u_\xi)_r\|_{L^1(\Tbb)}\, d\si_d(\xi)
&= \int_{\spd}\lim_{r\to 1-}\|(u_\xi)_r\|_{L^1(\Tbb)}\, d\si_d(\xi) \\
&= \lim_{r\to 1-} \|u_r\|_{L^1(\spd)} < \infty.
\end{aligned}
\end{equation}
Hence, for $\si_d$-a.e.\ $\xi\in\spd$, we have $\sup_{0<r<1} \|(u_\xi)_r\|_{L^1(\Tbb)} <\infty$.
Therefore, for $\si_d$-a.e.\ $\xi\in\spd$,
there exists $\mu_\xi\in M(\spd)$ such that $\textrm{supp\,}\mu_\xi \subset \Tbb\xi$ and
\[
u(z\xi) = \int_{\spd} \frac{1-|z|^2}{|w- z\xi|^2}\, d\mu_\xi(w), \quad z\in\Dbb.
\]
Observe that $\mu_\xi = \mu_{\lambda\xi}$ for $\xi\in\spd$, $\lambda\in\Tbb$.
So, the slice measure $\mu_\za$ is defined for $\wsd$-a.e. $\za\in\ppd$.
If $\mu$ is a positive pluriharmonic measure, then $\mu_\za$ is clearly defined for any $\za\in\ppd$.

Using \eqref{e_u_xi}, we conclude that
\begin{equation}\label{e_muz_norm}
\int_{\spd} \|\mu_\za\|\, d\wsd(\za) < \infty
\end{equation}
for any $\mu\in \plh (\spd)$.

The following result is formulated in \cite[Chapter~5, Section~3.1]{Aab85}.
For the reader's convenience, we supply certain details of the proof.

\begin{proposition}\label{p_slices}
  Let $\mu\in\plh(\spd)$ and let $\mu_\za$ denote the slice measure defined as above for $\wsd$-a.e.\ $\za\in\ppd$.
  Then
  \[
  \mu = \int_{\ppd} \mu_\za\, d\wsd(\za)
  \]
in the weak sense. In particular, $\plh(\spd) \subset \dcm(\spd)$.
\end{proposition}
\begin{proof}
Let $f\in C(\spd)$.
Given a measure $\mu\in \plh(\spd)$,
put $u = P[\mu]$.
By \eqref{e_Leb_dcm},
\[
\int_{\spd} f u_r\, d\si_d = \int_{\ppd} \int_{\spd} f u_r\, d\si_1^\za\, d\wsd(\za), \quad 0<r<1.
\]
Observe that $u_r \si_d \to \mu$ and $u_r \si_1^\za\to \mu_\za$
as $r\to 1-$
in $\sigma(M(\spd), C(\spd))$-topology for $\wsd$-a.e.\ $\za\in\ppd$.
So, taking the limit as $r\to 1-$
and applying \eqref{e_muz_norm}, we obtain
\[
\int_{\spd} f\, d\mu = \int_{\ppd} \int_{\spd} f \, d\mu_\za\, d\wsd(\za)
\]
by the dominated convergence theorem.
\end{proof}

\subsection{Properties of decomposable measures}
The main results of the present section are
Proposition~\ref{p_non_neg_muz_prop}, Theorems~\ref{t_mu_mod} and \ref{t_DM} and Proposition~\ref{p_DM_sgl}
below.
We start with auxiliary lemmas.

\begin{lemma}\label{l_non_neg}
Let $\mu\in\dcm(\spd)$ be a positive measure. Assume that
$\mu_\za$ is a real measure for $\wsd$-a.e. $\za\in\ppd$. Let $f\in C(\spd)$, $f\ge 0$.
Then
\[
\int_{\spd} f\, d\mu_\za\ge0\quad \textrm{for\ } \wsd\textrm{-a.e.\ } \za\in\ppd.
\]
\end{lemma}
\begin{proof}
Suppose that the exists a set $E\subset \ppd$ such that $\wsd(E) >0$ and
\[
\int_{\spd} f\, d\mu_\za < 0\quad \textrm{for all\ } \za\in E.
\]
Fix a compact set $K\subset E$ such that $\wsd(K)>0$.
Select a sequence $\{f_n\}_{n=1}^\infty$ in $C(\spd)$  such that $0\le f_n\le 1$
and
\[
\lim_{n\to\infty}f_n(\xi)= {\chi}_{\pi^{-1}(K)}(\xi) \quad \textrm{for all\ } \xi\in\spd,
\]
where ${\chi}_{\pi^{-1}(K)}$ is the characteristic function of the set ${\pi^{-1}(K)}$.
Applying \eqref{e_slice_sphsph_C_proj} and Remark~\ref{r_DM}, we obtain
\[
\int_{\ppd} \int_{\spd} f {\chi}_{\pi^{-1}(K)} \, d\mu_\za\, d\widehat\si_d(\za)
= \int_{\spd} f {\chi}_{\pi^{-1}(K)} \, d\mu
= \lim_{n\to\infty} \int_{\spd} f f_n \, d\mu
\ge 0,
\]
since $\mu$ is a positive measure.
However,
\begin{align*}
  \int_{\spd} f{\chi}_{\pi^{-1}(K)} \, d\mu_\za &= \int_{\spd} f \, d\mu_\za<0 \quad\textrm{for all\ } \za\in K; \\
  \int_{\spd} f {\chi}_{\pi^{-1}(K)} \, d\mu_\za &=0 \quad\textrm{for all\ } \za\in\ppd\setminus K.
\end{align*}
Since $\wsd(K)>0$, we have a contradiction.
So, the proof of the proposition is finished.
\end{proof}

\begin{lemma}\label{non_neg_muz}
Let $\mu\in\dcm(\spd)$ be a positive measure. Assume that
$\mu_\za$ is a real measure for $\wsd$-a.e. $\za\in\ppd$.
Then $\mu_\za\ge 0$ for $\wsd$-a.e. $\za\in\ppd$.
\end{lemma}
\begin{proof} Let $\Lambda$ be a countable dense subset of $\{f\in C(\spd):f\ge 0\}$.
Given an $f\in\Lambda$, Lemma~\ref{l_non_neg} provides a set $E_f\subset \ppd$ such that $\wsd(E_f)=0$
and
\[
\int_{\spd} f\mu_\za\ge 0\quad \textrm{for all\ } \za\in\ppd\setminus E_f.
\]
So,
\[
\int_{\spd} f\mu_\za\ge 0 \quad \textrm{for all\ } f\in\Lambda \textrm{\ and \ }\za\in\ppd \setminus \bigcup\limits_{f\in\Lambda}E_f.
\]
Hence,
$\mu_\za\ge0$
for all $\za\in\ppd \setminus \bigcup\limits_{f\in\Lambda}E_f$.
\end{proof}

\begin{corollary}\label{mu_nol}
Let \eqref{int_decom} hold with $\mu=0$.  Then $\mu_\za=0$ for $\wsd$-a.e. $\za\in\ppd$.
\end{corollary}
\begin{proof}
Consideration of the families $\{\Rl\,\mu_\za\}$ and $\{{\mathrm{Im}}\,\mu_\za\}$
allows to make the following additional assumption:
$\mu_\za$ is real for $\wsd$-a.e. $\za\in\ppd$.
Now, the corollary immediately follows from Lemma~\ref{non_neg_muz}.
\end{proof}

So, we have the following uniqueness property: if $\mu\in \dcm(\spd)$ and we are given decompositions
\[
\mu = \int_{\ppd} \mu_\za\, d\wsd (\za)= \int_{\ppd} \nu_\za\, d\wsd (\za)
\]
in the sense of \eqref{int_decom},
then $\mu_\za=\nu_\za$ for $\wsd$-a.e. $\za\in\ppd$.

\begin{proposition}\label{p_non_neg_muz_prop}
Let $\mu\in\dcm(\spd)$ be a positive measure.
Then $\mu_\za\ge 0$ for $\wsd$-a.e. $\za\in\ppd$.
\end{proposition}
\begin{proof}
Corollary~\ref{mu_nol} guarantees that ${\mathrm{Im}}\,\mu_\za=0$ for $\wsd$-a.e. $\za\in\ppd$.
So, Lemma~\ref{non_neg_muz} applies.
\end{proof}

\begin{theorem}\label{t_mu_mod} Let $\mu\in \dcm(\spd)$ and \eqref{int_decom} holds.
Then $|\mu|$ is also decomposable and
\[
|\mu| = \int_{\ppd} |\mu_\za|\, d\widehat\si_d (\za)
\]
in the weak sense.
\end{theorem}
\begin{proof}
Take a Borel function $h$ on $\spd$ such that $|h|=1$ everywhere and
$|\mu|=h\mu$. Then, for any $f\in C(\spd)$, we have
\[
\int_{\spd} fd|\mu| =\int_{\spd} fh \,d\mu=\int_{\ppd}\int_{\spd} f h\, d\mu_\za\, d\wsd (\za)
\]
by Remark~\ref{r_DM}.
Applying Proposition~\ref{p_non_neg_muz_prop} to the above decomposition of $|\mu|$,
we conclude that  $h\mu_\za\ge0$ for $\wsd$-a.e. $\za\in\ppd$.
It remains to observe that the property $h\mu_\za\ge 0$ implies $h\mu_\za=|\mu_\za|$
for $\wsd$-a.e. $\za\in\ppd$.
\end{proof}

\begin{lemma}\label{l_DM_zero}
Let $\mu\in\dcm(\spd)$ be a positive measure. Assume that
$E\subset \spd$ and $\mu(E)=0$.
Then $\mu_\za (E)=0$ for $\wsd$-a.e. $\za\in\ppd$.
\end{lemma}
\begin{proof}
Take a Borel set $E_0\subset \spd$ such that $E_0\supset E$ and $\mu(E_0)=0$.
Using Proposition~\ref{p_non_neg_muz_prop}, Remark~\ref{r_DM}
and applying \eqref{e_slice_sphsph_C_proj} with $f=\chi_{E_0}$,
we obtain $0\le \mu_\za(E)\le\mu_\za(E_0)=0$ for $\wsd$-a.e.\ $\za\in\ppd$.
\end{proof}

\begin{theorem}\label{t_DM}
Let $\mu$ be a decomposable measure on $\spd$ satisfying \eqref{int_decom}.
Then $f\mu\in \dcm(\spd)$  for any $f\in L^1(|\mu|)$; moreover,
\begin{equation}\label{e_DM}
f\mu = \int_{\ppd} f\mu_\za\, d\widehat\si_d (\za)
\end{equation}
in the weak sense.
\end{theorem}
\begin{proof}
By Theorem~\ref{t_mu_mod}, we may assume, without loss of generality, that $\mu\ge 0$.
Also, we may assume that $f\ge 0$.

Next, by Lemma~\ref{l_DM_zero}, it suffices to prove \eqref{e_DM}
under additional assumption that $f$ is a Borel function defined
everywhere on $\spd$.
Now, if $f$ is bounded, then \eqref{e_DM} holds by \eqref{e_slice_sphsph_C_proj} and Remark~\ref{r_DM}.
Finally, if
$f\ge 0$
is an unbounded Borel function, then
$f=\lim\limits_{n\to\infty}f_n$ everywhere for a monotonically increasing sequence
$\{f_n\}$ of bounded Borel functions
$f_n\ge 0$;
hence, \eqref{e_DM} holds.
\end{proof}

\begin{corollary}\label{c_DM_ll}
Let $\mu\in \dcm(\spd)$. Assume that $\nu\ll \mu$.
Then $\nu$ is also a decomposable measure.
\end{corollary}

To work with the singular parts of decomposable measures, we need the following lemma.

\begin{lemma}\label{l_sgl}
Let $\mu$ be a decomposable measure such that \eqref{int_decom} holds.
If $\mu\bot\si_d$, then $\mu_\za \bot \si_1^\za$
for $\wsd$-a.e.\ $\za\in\ppd$.
\end{lemma}
\begin{proof}
We may assume, without loss of generality, that $\mu$ is a positive measure.
Select a Borel set $E\subset \spd$ such that $\mu(E)=0$ and $\si_d(E) =1$.
Then $\si_1^\za(E) =1$ for $\wsd$-a.e.\ $\za\in\ppd$.
Also, $\mu_\za \ge 0$, hence, $\mu_\za(E)=0$ for $\wsd$-a.e.\ $\za\in\ppd$.
Therefore, $\mu_\za \bot \si_1^\za$ for $\wsd$-a.e.\ $\za\in\ppd$, as required.
\end{proof}

For $\mu\in \dcm(\spd)$, let $\mu^a$ ($\mu_\za^a$) denote the absolutely continuous part
of $\mu$ ($\mu_\za$) with respect to
Lebesgue measure $\si_d$ ($\si_1^\za$).
Let $\mu^s$ ($\mu_\za^s$) denote the corresponding singular part.

\begin{proposition}\label{p_DM_sgl}
Let $\mu$ be a decomposable measure such that \eqref{int_decom} holds.
Then
\[
\mu^s = \int_{\ppd} \mu_\za^s\, d\wsd(\za)
\]
in the weak sense.
\end{proposition}
\begin{proof}
Since $\mu^a \ll \si_d$, we have a decomposition
\[
\mu^a = \int_{\ppd} \nu_\za\, d\wsd(\za)
\]
in the weak sense and for certain measures $\nu_\za$, $\nu_\za \ll \si_1^\za$ for $\wsd$-a.e.\ $\za\in\ppd$.
Also, Corollary~\ref{c_DM_ll} guarantees that $\mu^s\in \dcm(\spd)$. Hence, by Lemma~\ref{l_sgl},
\[
\mu^s = \int_{\ppd} \rho_\za\, d\wsd(\za)
\]
for certain $\rho_\za$, $\rho_\za \bot \si_1^\za$ for $\wsd$-a.e.\ $\za\in\ppd$.
Since $\mu = \mu^a + \mu^s$ has a unique integral decomposition in the sense of \eqref{int_decom}, we conclude that $\mu^a_\za = \nu_\za$
and $\mu^s_\za = \rho_\za$ for $\wsd$-a.e.\ $\za\in\ppd$. In particular,
the required decomposition of $\mu^s$ holds.
\end{proof}

\subsection{Pluriharmonic measures and Cauchy integrals}
For $\mu\in M(\spd)$, the Cauchy transform $\mu_+$ is defined as
\[
  \mu_+(z) = \int_{\spd} C(z, \xi)\, d\mu(\xi), \quad z\in \bd.
\]
Also, put
\begin{equation}\label{e_mum}
  \mu_-(z) = \int_{\spd} (C(\xi, z) -1)\, d\mu(\xi), \quad z\in\bd.
\end{equation}
Observe that $\mu_+(z) + \mu_-(z) = P[\mu](z)$, $z\in \bd$, for all $\mu\in \plh(\spd)$.

It is known that the radial limit $\mu_+(\xi) = \lim_{r\to 1-} \mu_+(r\xi)$ exists
for $\si_d$-a.e. $\za\in\spd$.
For $y\ge 0$, let $\chi_{\{|\mu_+|>y\}}$ denote
the characteristic function of the set
\[
\{\xi\in \spd: |\mu_+ (\xi)|>y\}.
\]

For $d=1$, the following result was obtained in \cite{Po96}.

\begin{proposition}\label{p_polto}
Let $\mu\in\plh(\spd)$.
Then
\[
\pi y \chi_{\{|\mu_+|> y \}} \si_d \overset{w*}{\longrightarrow} |\mu^s|\quad \textrm{as } y\to +\infty.
\]
\end{proposition}
\begin{proof}
Let $\za\in \ppd$.
If the slice measure $\mu_\za \in M(\spd)$ is defined, then put
\[
\cau \mu_\za (r\xi) = \int_{\spd} \frac{1}{1- r\langle\xi, w\rangle}\, d\mu_\za(w), \quad \xi\in\pi^{-1}(\za),
\ 0<r<1.
\]
Since $\mu\in\plh(\spd)$, we have
\[
\mu_+ (r\xi) = \cau\mu_\za (r\xi), \quad \xi\in\pi^{-1}(\za),\ 0<r<1.
\]
Hence, given an $f\in C(\spd)$, the following equality holds:
\[
\int_{\spd} f \chi_{\{|\mu_+|> y \}}\, d\si_d =
\int_{\ppd} \int_{\spd} f \chi_{\{|\cau\mu_\za|> y \}}\, d\si_1^\za\, d\wsd(\za),
\]
where
$\chi_{\{|\cau\mu_\za|> y \}}$ is the characteristic function of the set
\[
\{\xi\in\spd: \pi(\xi)=\za\ \textrm{and\ } |\cau\mu_\za(\xi)|> y \}.
\]
By \cite[Theorem~1]{Po96},
\[
\pi y \chi\{\xi\in\pi^{-1}(\za): |\cau\mu_\za(\xi)|> y \} \si_1^\za
\overset{w*}{\longrightarrow} |\mu_\za^s|\quad \textrm{as } y\to +\infty.
\]
Combining the above properties, Proposition~\ref{p_slices}, Theorem~\ref{t_mu_mod}
and Proposition~\ref{p_DM_sgl}, we obtain
\[
\lim_{y\to +\infty} \pi y \int_{\spd} f \chi_{\{|\mu_+|> y \}}\, d\si_d
= \int_{\ppd} \int_{\spd} f \, d|\mu^s_\za| \, d\wsd(\za)
= \int_{\spd} f \, d|\mu^s|,
\]
as required.
\end{proof}

\begin{corollary}[see {\cite[Chapter~5, Section~3.2]{Aab85}}]\label{c_cauchy_aab}
Let $\mu\in\plh(\spd)$.
Then
\[
\lim_{y\to +\infty} \pi y \si_d\{\za\in\spd: |\mu_+(\za)|> y \} \to \|\mu^s\|.
\]
\end{corollary}

\subsection{Pluriharmonic and representing measures}
By definition, the ball algebra $A(\bd)$ consists of those $f\in C(\overline{\bd})$ which are holomorphic in $\bd$.
For $z\in \bd$, let $M_z(\spd)$ denote the set of those probability measures $\rho\in M(\spd)$
which represent the point $z$ for $A(\bd)$, that is,
\[
\int_{\spd} f\, d\rho = f(z) \quad \textrm{for all}\ f\in A(\bd).
\]
Elements of $M_z(\spd)$ are called representing measures.

\begin{df}\label{d_totnull}
A measure $\mu\in M(\spd)$ is said to be \textsl{totally singular} if $\mu\bot \rho$
for all $\rho\in M_0(\spd)$.
A set $E\subset\spd$ is called \textsl{totally null} if $\rho(E)=0$ for all $\rho\in M_0(\spd)$.
\end{df}

It is easy to check that the notions introduced in Definition~\ref{d_totnull}
do not change if $M_0(\spd)$ is replaced by $M_z(\spd)$ for any $z\in \bd$;
see, for example, \cite[Sect.~9.1.3]{Ru80}.

We will use the following theorem in Sections~\ref{s_clk_inner} and \ref{s_model}.

\begin{theorem}[{\cite[Theorem~10]{Dzap94}}]\label{t_aif}
Let $\mu\in\plh(\spd)$. Then $\mu^s$ is totally singular.
\end{theorem}

Observe that Corollary~\ref{c_cauchy_aab} plays an important role in the proof of the above result.
See also \cite{DAIF98} for generalizations of Theorem~\ref{t_aif}.

\section{Clark measures}\label{s_aux_clk}

\subsection{A disintegration theorem for Clark measures}
Proposition~\ref{p_slices} suggests that the Clark measures on the unit sphere inherit certain properties of
the classical Clark measures on the unit circle.
As an illustration, we prove an analog of the so-called disintegration theorem for $d\ge 2$;
see Theorem~\ref{t_desint} below.

The following lemma is standard, so we omit its proof.

\begin{lemma}\label{l_MS_wlim}
Let $\{\mu_n\}_{n=1}^\infty$ be a bounded sequence
in $M(\spd)$ and let $\mu\in M(\spd)$. Then $\lim\limits_{n\to\infty}\mu_n=\mu$
in $\clk(M(\spd), C(\spd))$-topology if and only if
$\lim\limits_{n\to\infty} {P}[\mu_n](z)= {P}[\mu](z)$ for all $z\in\bd$.
\end{lemma}

\begin{corollary}\label{c_slices_wcont}
Let $\ph: \bd\to \Dbb$ be a holomorphic function. The
mapping $\al\mapsto\clk_\al[\ph]$
is continuous from $\Tbb$ into the space $M(\spd)$ endowed with the weak topology.
\end{corollary}

The following theorem is obtained in \cite{Aab87} for $d=1$.

\begin{theorem}\label{t_desint}
Let $\ph: \bd\to\Dbb$ be a holomorphic function and
let $\clk_\al = \clk_\al[\ph]$, $\al\in \Tbb$.
Then
\[
\int_{\Tbb}
\int_{\spd}  f\, d\clk_\al\, d\si_1(\al) = \int_{\spd} f\, d\si_d
\]
for all $f\in C(\spd)$.
\end{theorem}

\begin{proof}
By Corollary~\ref{c_slices_wcont}, the functional
\[
f\mapsto\int_{\Tbb}\left(\int_{\spd} f \,d\clk_\al\right)\,d\si_1(\al)
\]
is defined for all $f\in C(\spd)$. Clearly, the functional is continuous on $C(\spd)$.
Hence, there exists a measure $\mu\in M(\spd)$
such that
\[
\mu=\int_{\Tbb}\clk_\al\,d\si_1(\al)
\]
in the weak sense.
Now, observe that
\begin{align*}
\int_{\spd} P(z,\za)\,d\mu(\za)
&=
\int_{\Tbb}\left(\int_{\spd}P(z,\za)\, d\clk_\al(\za)\right)\, d\si_1(\al) \\
&=
\int_{\Tbb}\Rl \left(\frac{\al+ \ph(z)}{\al-\ph(z)} \right)\,d\si_1(\al) \\
&= 1 \\
&= \int_{\spd} P(z,\za)\,d\si_d(\za)
\end{align*}
for all $z\in\bd$.
Hence, $\mu=\si_d$.
\end{proof}

\begin{rem}
Proposition~\ref{p_slices} and Fubini's theorem allow to deduce Theorem~\ref{t_desint} for $d\ge 2$
from the corresponding result for $d=1$.
\end{rem}

\subsection{Absolutely continuous and singular parts of Clark measures}
Given an $\al\in\Tbb$
and a holomorphic function $\ph: \bd\to \Dbb$,
let $\clk_\al^a$ denote the absolutely continuous part
of the Clark measure $\clk_\al = \clk_\al[\ph]$.
The definition of $\clk_\al$ and basic properties of Poisson integrals guarantee that
\[
d\clk_\al^a (\za)= \frac{1-|\ph(\za)|^2}{|\al - \ph(\za)|^2}\, d\si_d(\za),
\]
where $\ph(\za) = \lim_{r\to 1-} \ph(r\za)$. Recall that $\ph(\za)$ is defined
for $\si_d$-a.e.\ $\za\in\spd$.

Since $\clk_\al$ is a positive measure, we have
\[
\|\clk_\al\| = P[\clk_\al](0) = \frac{1-|\ph(0)|^2}{|\al - \ph(0)|^2}.
\]
Therefore,
\begin{equation}\label{e_sclk_norm}
\frac{1-|\ph(0)|^2}{|\al - \ph(0)|^2} -
\int_{\spd}\frac{1-|\ph(\za)|^2}{|\al - \ph^(\za)|^2}\, d\si_d(\za) = \|\clk_\al^s\|,
\end{equation}
where $\clk_\al^s$ denotes the singular part of $\clk_\al$.

\subsection{Clark measures and Cauchy kernels}
The following lemma is a particular case of Theorem~1
from \cite[Chap.~V, \S21, Sect.~66]{Shab92}.

\begin{lemma}\label{l_diag}
Let $F$ be a holomorphic function on $\bd\times \bd$.
If $F(z, \overline{z}) =0$ for all $z\in\bd$,
then $F(z, w) =0$ for all $(z, w) \in \bd\times \bd$.
\end{lemma}

\begin{proposition}\label{p_cauchy_dbl}
Let $\ph: \bd\to\Dbb$, $d\ge 2$, be a holomorphic function and
let $\clk_\al = \clk_\al[\ph]$, $\al\in \Tbb$.
 Then
  \[
  \int_{\spd} C(z, \za) C(\za, w)\, d\clk_\al(\za) =
  \frac{1- \ph(z)\overline{\ph(w)}}{(1-\overline{\al}{\ph(z)})(1-\al\overline{\ph(w)})} C(z,w)
  \]
for all $\al\in\Tbb$, $z, w \in\bd$.
\end{proposition}
\begin{proof}
  The equality
  \[
  \int_{\spd} P(z, \za) \, d\clk_\al(\za) = \frac{1-|\ph(z)|^2}{|\al- \ph(z)|^2}, \quad z\in \bd,
  \]
  and the definition of $P(z,\za)$ guarantee that
   \[
  \int_{\spd} C(z, \za) C (\za, z)\, d\clk_\al(\za) = \frac{1-|\ph(z)|^2}{|\al- \ph(z)|^2} C(z,z), \quad z\in \bd.
  \]
  It remains to apply Lemma~\ref{l_diag}.
\end{proof}

\begin{corollary}\label{c_cauchyof_clk}
Let $\ph: \bd\to\Dbb$, $d\ge 2$, be a holomorphic function.
Then
\[
\int_{\spd} C(z, \za)\, d\clk_\al[\ph](\za) = \frac{1}{1-\overline{\al} \ph(z)} + \frac{\al\overline{\ph(0)}}{1-\al\overline{\ph(0)}}
\]
for all $\al\in\Tbb$, $z\in\bd$.
\end{corollary}
\begin{proof}
  Apply Proposition~\ref{p_cauchy_dbl} with $w=0$.
\end{proof}

\section{Clark measures and inner functions}\label{s_clk_inner}
In this section, we restrict our attention to the case, where $\ph$ is an inner function.
So, we use the symbol $I$ in the place of $\ph$.

\subsection{Total singularity and Clark measures}
As indicated in the introduction, if $I: \bd\to \Dbb$ is inner, then
$\clk_\al=\clk_\al[I]$ is singular for any $\al\in\Tbb$.
Moreover, by the following lemma, $\clk_\al$ is totally singular in the sense of Definition~\ref{d_totnull}.

\begin{lemma}\label{l_tot_sng}
Let $I$ be an inner function in $\bd$, $d\ge 2$.
Then $\clk_\al = \clk_\al[I]$ is totally singular for any $\al\in\Tbb$.
\end{lemma}
\begin{proof}
Fix an $\al\in\Tbb$ and put
  \[
  E = \left\{\za\in\spd: \lim_{r\to 1-}\frac{1-|I(r\za)|^2}{|\al-I(r\za)|^2} =+\infty \right\}.
  \]
Since $\clk_\al$ is a positive singular measure, we have $\clk_\al(\spd\setminus E) = 0$.
Observe that $E\subset E_0 = \{\za\in\spd: \lim_{r\to 1-} I(r\za) =\al\}$.
By \cite[Theorem~9.3.2]{Ru80}, $E_0$ is a totally null set.
Therefore, $\clk_\al$ is totally singular.
\end{proof}

\begin{rem}
  Since $\clk_\al = \clk_\al[I]$, $\al\in\Tbb$, is a singular pluriharmonic measure, Theorem~\ref{t_aif}
  also guarantees that $\clk_\al$ is totally singular.
\end{rem}

\subsection{The ball algebra $A(\bd)$ and $L^2(\clk_\al)$}
We will need the following classical notion.

\begin{df}[see {\cite[Sect.~9.1.5]{Ru80}}]\label{d_Henkin}
We say that $\mu\in M(\spd)$ is a \textsl{Henkin measure} if
\[
\lim_{j\to\infty} \int_{\spd} f_j\, d\mu =0
\]
for any bounded sequence $\{f_j\}_{j=1}^\infty \subset A(\bd)$ with the following property:
\[
\lim_{j\to\infty} f_j(z) = 0 \quad\textrm{for any\ } z\in\bd.
\]
\end{df}

\begin{lemma}\label{l_AB_dense}
Let $I$ be an inner function in $\bd$ and let $\clk_\al = \clk_\al[I]$, $\al\in\Tbb$.
Then the ball algebra $A(\bd)$ is dense in $L^2(\clk_\al)$.
\end{lemma}
\begin{proof}
Assume that $A(\bd)$ is not dense in $L^2(\clk_\al)$. Then there exists a non-trivial function $h\in L^2(\clk_\al)$
such that $h\clk_\al \in A(\bd)^\bot$, that is,
\[
\int_{\spd} f h\, d\clk_\al =0\quad \textrm{for all}\ f\in A(\bd).
\]
 So, $h\clk_\al$ is clearly
 a Henkin measure.
Hence, by the Cole--Range theorem (see \cite{CR72} or \cite[Theorem~9.6.1]{Ru80}), $h\clk_\al\ll\rho$
for some representing measure $\rho\in M_0(\spd)$.
However, $h\clk_\al\bot \rho$ by Lemma~\ref{l_tot_sng}. This contradiction finishes the proof of the lemma.
\end{proof}

\subsection{Inner functions with additional properties}
The definition of an inner function in $\bd$, $d\ge 2$,
is based on existence of the corresponding radial limits $\si_d$-a.e.
However, given an $f\in H^\infty(\bd)$, the slice function $f_\za(\lad) = f(\lad\za)$, $\lad\in \Dbb$,
has radial limits $\si_1$-a.e.
In fact, if $I$ is an arbitrary inner function in $\bd$, then
$I_\za$ is clearly an inner function in $\Dbb$ for $\si_d$-a.e.\ $\za\in\spd$.
This observation leads to further questions and results.
On the one hand, given a $\za\in\spd$,
there exists an inner function $I$ such that $I_\za$ is not inner.
Moreover, the following theorem holds.

\begin{theorem}[{\cite[Corollary~1 after Theorem~4]{Aab82}}]\label{t_inner_subball}
Let $d, n\in \Nbb$, $d>n$.
Given a holomorphic function $g: B_n \to \Dbb$, there exists an inner function $I$ in $\bd$
such that $I(z, 0) = g(z)$ for all $z\in\bd$.
\end{theorem}

On the other hand, we have the following result,
which gives a positive answer to a question asked by W.~Rudin \cite[Sect.~19.1]{Ru86}.

\begin{theorem}[\cite{Dup88}; see also \cite{DJFA2k}]\label{t_dup}
Let $d\ge 2$. There exists an inner function $I$ in $\bd$ such that
$I_\za(\lad):= I(\lad\za)$, $\lad\in\Dbb$, is an inner function in $\Dbb$ for all $\za\in\spd$.
\end{theorem}

Further pursuing the above idea and having in mind Lemma~\ref{l_tot_sng}, we naturally arrive at
a hypothetical function $I$
such that $I$ is, in an appropriate sense, inner with respect to any Henkin measure.
Indeed, given an $f\in H^\infty(\bd)$ and a Henkin measure $\mu\in M(\spd)$,
there exists $f[\mu] \in L^\infty(|\mu|)$ such that
\[
\lim_{r\to 1-} \int_{\spd} f_r(\za) g(\za)\, d|\mu|(\za) = \int_{\spd} f[\mu](\za) g(\za)\, d|\mu|(\za)
\]
for any $g\in L^1(|\mu|)$; see \cite[Sect.~11.3.1]{Ru80}.
In other words, $\lim_{r\to 1-} f_r = f[\mu]$ in $\sigma(L^\infty(|\mu|), L^1(|\mu|))$-topology.
So, we have the following problem:

\begin{prb}[{see \cite[Problem~1]{Aab86}}]\label{p_tot_in}
Let $d\ge 2$. Does there exist an inner function $I: \bd \to \Dbb$
such that $|I[\mu]| =1$ $|\mu|$-a.e.\ for any Henkin measure $\mu\in M(\spd)$?
\end{prb}

\begin{rem}
If we replace the set of Henkin measures by $M_0(\spd)$ in the above problem, then
we obtain exactly the same question.
Indeed, assume that $I: \bd \to \Dbb$ is a holomorphic function such that
$|I[\rho]| =1$ $\rho$-a.e.\ for any $\rho\in M_0(\spd)$.
Let $\mu$ be a Henkin measure. By the Cole--Range theorem, $\mu\ll\rho_0$
for some representing measure $\rho_0\in M_0(\spd)$.
Therefore, $|I[\mu]|=1$ $|\mu|$-a.e.
\end{rem}

However, any inner function $I$ is far from having the property required in Problem~\ref{p_tot_in};
see Corollary~\ref{c_wlim_inner_M0}.
For the sake of completeness, the corresponding arguments are presented in the next subsection.

\subsection{Inner functions with additional properties: proofs}
The main result of this subsection is the following theorem:

\begin{theorem}\label{t_wlim_M0}
Let $f\in H^\infty(\bd)$, $d\ge 2$.
Then there exists a set $\Lambda\subset \Cbb$ such that $\Lambda$ is dense in $f(\bd)$ and the following property holds:
For any $\ant\in\Lambda$, there exists a representing measure $\rho\in \cup_{z\in \bd} M_z(\spd)$ such that
$\lim_{r\to 1-} f_r =\ant$ in $\sigma(L^\infty(\rho), L^1(\rho))$-topology.
\end{theorem}

If $I$ is an inner function in $\bd$, $d\ge 2$,
then $I(\bd)$ is known to be dense in $\Dbb$.
So, we also have the following corollary:

\begin{corollary}\label{c_wlim_inner_M0}
Let $I$ be an inner function in the unit ball $\bd$, $d\ge 2$.
Then there exists a set $\Lambda\subset \Dbb$ such that $\Lambda$ is dense in $\Dbb$ and the following property holds:
For any $\ant\in\Lambda$, there exists a representing measure $\rho\in \cup_{z\in \bd} M_z(\spd)$ such that
$I[\rho] =\ant$ $\rho$-a.e.
\end{corollary}

Clearly, the above corollary guarantees that the answer to Problem~\ref{p_tot_in} is by far negative.

The rest of the present subsection is devoted to the proof of Theorem~\ref{t_wlim_M0}.
Observe that it suffices to prove Theorem~\ref{t_wlim_M0} for $d=2$.
So, for an appropriate point $\ant\in f(B_2)$, we will consider the connected component $\Mfld$
of the set $\{w\in B_2: f(w)=\ant\}$.
Namely, applying Sard's theorem, we select $\ant$ such
that $\Mfld$ is a one-dimensional complex submanifold of the ball $B_2$.

Before giving a proof of Theorem~\ref{t_wlim_M0},
we obtain auxiliary Lemmas~\ref{l_inner_c2} and \ref{l_linf_Phi} below in a more
general setting, assuming that $\Mfld$ is an arbitrary connected one-dimensional complex submanifold of $B_2$.
Observe that the topological boundary $\partial \Mfld$ of $\Mfld$ is a subset of the sphere $\sptwo$.
Next, let $\Phi: U\to \Mfld$ be a universal covering map, where $U$ is a simply connected complex manifold.
Since there exists a non-constant bounded holomorphic function on $U$ (in particular, $\Phi$ has this property),
the simply connected Riemann surface $U$ is hyperbolic.
Therefore, without loss of generality, we may suppose that
$U=\Dbb$; see, for example, \cite[Sect.~9-1]{Sp57}.

So, in what follows, we consider a universal covering map
$\Phi: \Dbb\to \Mfld$. For $\xi\in \Tbb$, let $\Phi(\xi)$ denote $\lim_{r\to 1-} \Phi(r\xi)$ if this limit exists.

\begin{lemma}\label{l_inner_c2}
Let $\Phi: \Dbb\to \Mfld$ be a universal covering map.
If $\Phi(\xi)$ is defined for certain $\xi\in \Tbb$, then $\Phi(\xi) \in \sptwo$.
In particular, $|\Phi(\xi)|=1$ for $\si_1$-a.e.\ $\xi\in \Tbb$.
\end{lemma}
\begin{proof}
Suppose that $\xi\in \Tbb$, $\Phi(\xi)$ is defined and $\Phi(\xi) \notin \sptwo$.
Then $\Phi(\xi) \in \Mfld$; hence, $\Phi$ is not a covering map.
This contradiction proves the claim.

Since $\Phi(\xi)$ is defined for $\si_1$-a.e.\ $\xi\in \Tbb$, the proof is finished.
\end{proof}

Put
\[
\widetilde{\Tbb} = \{\xi\in\Tbb: \lim_{r\to 1-} \Phi(r\xi)\ \textrm{exists}\}.
\]
Let $\mu$ be a complex Borel measure on $\Tbb$ such that $|\mu|(\Tbb\setminus \widetilde{\Tbb}) = 0$.
Given a Borel set $E\subset\sptwo$, define
\[
\mu_{\Phi}(E) = \mu(\Phi^{-1} (E)).
\]
So, $\mu_\Phi$ is a measure supported by $\partial \Mfld \subset \sptwo$.
We have
\begin{equation}\label{e_lift_Phi}
\int_{\sptwo} h(\za)\, d\mu_\Phi (\za) = \int_{\Tbb} h(\Phi(\xi))\, d\mu(\xi)
\end{equation}
for $h = \chi_E$, where $\chi_E$ denotes the characteristic function of a Borel set $E\subset\sptwo$.
Therefore, \eqref{e_lift_Phi} holds for any bounded Borel function $h$ on $\sptwo$.

Lemma~\ref{l_inner_c2} guarantees that $\si_1(\Tbb\setminus \widetilde{\Tbb}) = 0$.
So, applying the above procedure, put
\begin{equation}\label{e_Phiz_def}
\mPhi^\lad = \left( P(\lad, \cdot)\si_1 \right)_{\Phi}, \quad \lad\in \Dbb,
\end{equation}
where
\[
P(\lad, \xi) =\frac{1-|\lad|^2}{|\lad-\xi|^2}, \quad \lad\in\Dbb,\ \xi\in\Tbb,
\]
is the one-dimensional Poisson kernel.

By \eqref{e_lift_Phi}, we have
\begin{equation}\label{e_Phi_repr}
f(\Phi(\lad)) = \int_{\sptwo} f(\za)\, d\mPhi^\lad(\za), \quad \lad\in \Dbb,
\end{equation}
for all $f\in A(B_2)$.
Hence, $\mPhi^\lad$ is a representing measure, namely, $\mPhi^\lad \in M_{\Phi(\lad)}(\spd)$, $\lad\in\Dbb$.

Put $\mPhi= \mPhi^0$.
It is clear that $L^\infty (\mPhi^\lad)= L^\infty (\mPhi)$ for all $\lad\in \Dbb$.

\begin{lemma}\label{l_linf_Phi}
Let $h\in L^\infty (\mPhi)$ and $\ant\in \Cbb$.
Assume that
\[
\int_{\sptwo} h(\za)\, d\mPhi^\lad (\za) =\ant
\]
for all $\lad\in \Dbb$. Then $h=\ant$ $\mPhi$-a.e.\ on $\sptwo$.
\end{lemma}
\begin{proof}
By \eqref{e_lift_Phi} and \eqref{e_Phiz_def},
\[
\int_{\Tbb} h(\Phi(\xi)) P(\lad, \xi)\, d\si_1(\xi) = \int_{\sptwo} h(\za)\, d\mPhi^\lad(\za) =\ant
\]
for all $\lad\in\Dbb$.
Therefore, $h(\Phi(\xi)) = \ant$ for $\si_1$-a.e.\ $\xi\in\Tbb$.
So, we conclude that $h=\ant$ $\mPhi$-a.e.\ on $\sptwo$.
\end{proof}

Now, we are in position to prove the main result of the present subsection.

\begin{proof}[Proof of Theorem~\ref{t_wlim_M0}]
As mentioned above, it suffices to prove the theorem for $d=2$.
Also, without loss of generality, we may assume that $f$ is not a constant.

Let $E\subset B_2$ denote the set of critical points for $f$.
By Sard's theorem, $f(E)$ is a set of zero area 
measure.
Put $\Lambda = f(B_2)\setminus f(E)$.
We claim that $\Lambda$ has the required properties.

Indeed, fix a point $\ant\in\Lambda$.
Let $\Mfld$ denote a connected component
of the set $\{w\in B_2: f(w)=\ant\}$.
Since all points of $\Mfld$ are not critical for $f$,
$\Mfld$ is a one-dimensional complex submanifold of $B_2$.
So, applying the procedure described earlier in the present subsection,
define $\si_\Phi^\lad = \si_\Phi^\lad(\ant, \Mfld)$, $\lad\in \Dbb$, by \eqref{e_Phiz_def}.
Also, put $\mPhi= \mPhi^0$.

Since $\mPhi$ is a representing measure, there exists $h\in L^\infty(\mPhi)$ such that
$h = \lim_{r\to 1-} f_r$ in $\sigma(L^\infty(\mPhi), L^1(\mPhi))$-topology;
see \cite[Sect.~11.3.1]{Ru80}. For $\lad\in\Dbb$, the measure $\si_\Phi^\lad$
is absolutely continuous with respect to $\si_\Phi$.
Hence, using the defining property of $h$ and equality~\eqref{e_Phi_repr}, we obtain
\[
\int_{\sptwo} h\, d\mPhi^\lad = \lim_{r\to 1-} \int_{\sptwo} f_r\, d\mPhi^\lad
=\lim_{r\to 1-} f(r\Phi(\lad)) = f(\Phi(\lad))
=\ant
\]
for all $\lad\in \Dbb$.
Therefore, $h=\ant$ $\mPhi$-a.e.\ on $\sptwo$ by Lemma~\ref{l_linf_Phi}.
\end{proof}

\section{Two analogs of model spaces}\label{s_model}
For an inner function $\diin$ on $\Dbb$, the classical model space $K_\diin = K_\diin(\Dbb)$
is defined as $K_\diin = H^2(\Dbb) \ominus \diin H^2(\Dbb)$.
Given an inner function $I$ in $\bd$, $d\ge 2$, we consider the following analogs of the model space:
\begin{align*}
  \ksm &= \{f\in H^2:\, I\overline{f} \in H^2_0\};\\
  \kla &= H^2 \ominus I H^2,
\end{align*}
where $H^2= H^2(\bd)$. Clearly, $\ksm \subset \kla$.

Let $\al\in\Tbb$. In the present section, we construct a unitary operator
$U_\al$ from $\kla$ onto $L^2(\clk_\al)$; see Theorem~\ref{t_kla} below.
Next, in Section~\ref{ss_t11prf}, we prove Theorem~\ref{t_ksm_plh}, that is,
we characterize those $f\in L^2(\clk_\al)$ for which $U_\al^* f \in \ksm$.
So, as mentioned above, we use standard facts of the function theory in $\bd$ without explicit references.
In particular, we identify the Hardy space $H^p(\bd)$, $p>0$, and the space
$H^p(\spd)$ of the corresponding boundary values.

\subsection{A unitary operator from $\kla$ onto $L^2(\clk_\al)$}
Observe that
\[
K(z, w) \overset{\textrm{def}}{=} \frac{1-I(z)\overline{I(w)}}{(1-\langle z, w \rangle)^n}
= (1-I(z)\overline{I(w)}) C(z, w)
\]
is the reproducing kernel for $\kla$, that is, 
\[
g(z) = \int_{\spd} g(w) K(z, w) \, d\si_d(w), \quad z\in\bd,
\]
for all $g\in \kla$.
Indeed, $C(z, w)$ is the reproducing kernel for $H^2(\bd)$; hence,
$I(z) C(z, w) \overline{I(w)}$ is the reproducing kernel for $I H^2(\bd)$. Therefore,
the difference $C(z,w) - I(z) C(z, w) \overline{I(w)}$
is the reproducing kernel for $H^2(\bd) \ominus I H^2(\bd)$.

Put
$K_w(z) = K(z, w)$ and define
\[
(U_\al K_w)(\za) \overset{\textrm{def}}{=} \frac{1-\al \overline{I(w)}}{(1-\langle \za, w \rangle)^n}
= (1-\al \overline{I(w)}) C(\za, w), \quad \za\in\spd.
\]

\begin{theorem}\label{t_kla}
For each $\al\in\Tbb$, $U_\al$ has a unique extension
to a unitary operator from $\kla$ onto $L^2(\clk_\al)$.
\end{theorem}
\begin{proof}
Fix an $\al\in\Tbb$. Since $K(z, w)$ is the reproducing kernel function for $\kla$,
the linear span of the family $\{K_w\}_{w\in\bd}$ is dense in $\kla$.
Therefore, if the required extension exists, then it is unique.

  Now, we claim that $(U_\al K_w, U_\al K_z)_{L^2(\clk_\al)}=(K_w, K_z)_{H^2}$ for $z, w \in \bd$.
Indeed, applying Proposition~\ref{p_cauchy_dbl}, we obtain
  \begin{align*}
    (U_\al K_w, U_\al K_z)_{L^2(\clk_\al)}
    &=\int_{\spd} (1-\al \overline{I(w)}) C(\za, w) (1-\overline{\al} I(z)) C(z, \za)\, d\clk_\al(\za) \\
    &=(1-\al \overline{I(w)})(1-\overline{\al} I(z)) \int_{\spd} C(\za, w) C(z, \za)\, d\clk_\al(\za) \\
    &= (1- I(z) \overline{I(w)}) C(z, w)\\
    &= K(z,w) = (K_w, K_z)_{H^2}.
  \end{align*}
So, $U_\al$ extends to an isometric embedding of $\kla$ into $L^2(\clk_\al)$.
Hence, to finish the proof, it remains to observe that the linear span of the family $\{C(\za, z)\}_{z\in\bd}$
is dense in $L^2(\clk_\al)$ by Lemma~\ref{l_AB_dense}.
\end{proof}

\subsection{Image of $\ksm$}\label{ss_t11prf}
In this section, we prove Theorem~\ref{t_ksm_plh}.
Recall that $\mu_+$ denotes the Cauchy transform of a measure $\mu\in M(\spd)$
and
\[
\mu_+(z) + \mu_-(z) = P[\mu](z), \quad z\in \bd,
\]
for all $\mu\in \plh(\spd)$,
where $\mu_-$ is defined by \eqref{e_mum}.

We claim that
\begin{equation}\label{e_ual_conj}
\begin{aligned}
(U_\al^* f)(z)
    &= (1-\overline{\al}I(z)) \int_{\spd} C(z, \za) f(\za) \, d\clk_\al(\za) \\
    &  = (1-\overline{\al}I(z))(f\clk_\al)_+(z), \quad z\in\bd,
\end{aligned}
\end{equation}
for $f\in L^2(\clk_\al)$, $\al\in \Tbb$.
Indeed, the definition of $U_\al$ and Proposition~\ref{p_cauchy_dbl} imply the above equality
for $f(\za) = (1-\al \overline{I(w)}) C(\za, w)$ with $w\in\bd$.
By Lemma~\ref{l_AB_dense},
the linear span of the family
\[
\left\{(1-\al\overline{I(w)}) C(\za, w)\right\}_{w\in\bd}
\]
is dense in $L^2(\clk_\al)$. So, the claim is proved.

\begin{proof}[Proof of Theorem~\ref{t_ksm_plh}]
(ii)$\Rightarrow$(i)
Let $f\clk_\al\in\plh(\spd)$. Put $\overline{G} =- (f\clk_\al)_{-}$.
Then $G\in H^p_0$, $0<p<1$. The property $f\clk_\al\in\plh(\spd)$
guarantees that
\[
P[f\clk_\al](z) = (f\clk_\al)_{+}(z) -\overline{G}(z),\quad z\in\bd.
\]
Now, we consider the corresponding functions on $\spd$.
Since $f\clk_\al$ is a singular measure, we have $(f\clk_\al)_{+}(\za) = \overline{G}(\za)$
for $\si_d$-a.e.\ $\za\in\spd$.
Therefore,
\eqref{e_ual_conj} and Theorem~\ref{t_kla} imply that
\[
(1-\overline{\al}I) \overline{G} = U_\al^* f \in H^2(\spd).
\]
Hence, for $0<p<1$, we have
\[
I \overline{U_\al^* f} = I (1- \al\overline{I}) G = (I-\al)G \in L^2(\spd) \cap H^p_0(\spd) = H^2_0(\spd).
\]
So, (ii) implies (i).

(i)$\Rightarrow$(ii)
Let $F= U_\al^* f \in \ksm$.
For $\za\in\spd$, the slice function $F_\za$ is defined as $F_\za(z) = F(z\za)$, $z\in\Dbb$.
Observe that the following properties hold for $\si_d$-a.e.\ $\za\in\spd$:
\begin{itemize}
  \item $F_\za\in H^2 = H^2(\Dbb)$;
  \item $I_\za$ is an inner function in $\Dbb$;
  \item $F_\za\in (I_\za)_*(H^2) = (I_\za)_*(H^2(\Dbb))$.
\end{itemize}
In what follows, we assume that the point $\za$ under consideration satisfies the above conditions.

By \eqref{e_ual_conj}, we have $(1-\overline{\al} I(z))^{-1} F(z)\in H^p$, $0<p<1$.
By assumption, there exists $g\in H^2_0$ such that $F =I \overline{g}$.
Since $I$ is inner, we obtain
\[
\frac{\overline{g(\za)}}{\overline{I(\za)} -\overline{\al}}
=\frac{I(\za)\overline{g(\za)}}{1-\overline{\al} I(\za)}
=\frac{F(\za)}{1-\overline{\al} I(\za)}
\quad \textrm{for $\si_d$-a.e. } \za\in\spd.
\]
Observe that $({I} - {\al})^{-1} \in {H^p}$ for $0<p<1$
and $g\in H^2_0$, thus,
\[
G:= \frac{g}{I-{\al}} \in H^p_0
\]
for sufficiently small $p>0$, hence, for $0<p<1$.
Since
\[
(1-\overline{\al} I(\za))^{-1} F(\za) = \overline{G(\za)}\quad \textrm{for $\si_d$-a.e. } \za\in\spd,
\]
we have $G_\za\in H^p_0(\Dbb)$ and
\begin{equation}\label{e_Gzeta}
(1-\overline{\al} I_\za)^{-1} F_\za = \overline{G_\za}\quad \textrm{a.e.\ on\ } \Tbb
\end{equation}
for $\si_d$-a.e.\ $\za\in\spd$.

So, we consider $\za\in\spd$ such that the above property also holds.
Recall that $I_\za$ is an inner function. Let $\clk_\al[I_\za]$ denote the corresponding Clark measure on $\Tbb$.
Since $F_\za\in (I_\za)_*(H^2(\Dbb))$, the
Clark--Poltoratski theory in the unit disk (see \cite{Po93})
guarantees that
$F_\za \in L^2(\clk_\al[I_\za])$,
in the sense of boundary values, and the following representation holds:
\begin{equation}\label{e_cau_Fzeta}
(1-\overline{\al} I_\za)^{-1} F_\za(z) = \int_{\Tbb} C(z, \xi) F_\za(\xi)\, d\clk_\al[I_\za](\xi), \quad z\in \Dbb.
\end{equation}
Let $\mu_\za = F_\za \clk_\al[I_\za]$.
The measure $\mu_\za$ is singular, so $(\mu_\za)_+ + (\mu_\za)_- =0$ a.e.\ on $\Tbb$.
Thus $(\mu_\za)_- = - \overline{G_\za}$ a.e.\ on $\Tbb$ by \eqref{e_Gzeta} and \eqref{e_cau_Fzeta}.
Also, $(\mu_\za)_- \in \overline{H^p_0}$ and $- \overline{G_\za}\in \overline{H^p_0}$, $0<p<1$;
hence, $(\mu_\za)_- = - \overline{G_\za}$ on $\Dbb$.
Therefore,
\begin{equation}\label{e_slice_mu}
(1-\overline{\al} I_\za)^{-1} F_\za(z) - \overline{G_\za}(z)
= (\mu_\za)_+(z) + (\mu_\za)_-(z)
= P[\mu_\za](z), \quad z \in\Dbb.
\end{equation}
So, for $\si_d$-a.e. $\za\in\spd$,
\[
\int_{\Tbb} \left| \frac{F(r\xi\za)}{1-\overline{\al}I(r\xi\za)} - \overline{G(r\xi\za)}\right|\, d\si_1(\xi)
\le \|\mu_\za\|
\]
for all $0<r<1$. Integrating the above estimate with respect to $\za\in\spd$, we obtain
\begin{equation}\label{e_int_muza}
\int_{\spd} \left| \frac{F(r\za)}{1-\overline{\al}I(r\za)} - \overline{G(r\za)}\right|\, d\si_d(\za)
\le \int_{\spd} \|\mu_\za\|\, d\si_d(\za)
\end{equation}
for all $0<r<1$.

Since $\|F_\za\|_{L^2(\clk_\al[I_\za])}= \|F_\za\|_{H^2}$,
we have
\[
\|\mu_\za\|
=\|F_\za\|_{L^1(\clk_\al[I_\za])} \le \|F_\za\|_{L^2(\clk_\al[I_\za])} \sqrt{\|\clk_\al[I_\za]\|}
\le \|F_\za\|_{H^2} \sqrt{\frac{1+|I(0)|}{1-|I(0)|}}.
\]
Thus,
\[
\int_{\spd} \|\mu_\za\|\, d\si_d(\za) \le \|F\|_{H^2} \sqrt{\frac{1+|I(0)|}{1-|I(0)|}} < \infty.
\]
Therefore, by \eqref{e_int_muza}, there exists a measure $\nu\in\plh(\spd)$ such that
\begin{equation}\label{e_nu}
P[\nu]= (1-\overline{\al}I)^{-1} F - \overline{G}.
\end{equation}
Clearly, $\nu$ is singular with respect to $\si_d$.

Now, using \eqref{e_ual_conj} and \eqref{e_nu}, observe that $\overline{f}\clk_\al -\overline{\nu} \in {A(\bd)}^\bot$, thus $\overline{f}\clk_\al -\overline{\nu}$
is a Henkin measure. Hence, by the Cole--Range theorem,
\begin{equation}\label{e_ac_rep}
{f}\clk_\al - {\nu} \ll \rho
\end{equation}
for some representing measure $\rho$. By Lemma~\ref{l_tot_sng}, $f\clk_\al$ is totally singular;
by Theorem~\ref{t_aif}, $\nu$ is also totally singular because $\nu$ is a singular pluriharmonic measure.
So, ${f}\clk_\al - {\nu}$ is a totally singular measure and \eqref{e_ac_rep} holds. Therefore,
$f\clk_\al = \nu \in \plh(\spd)$; in particular, $f\clk_\al$ is a pluriharmonic measure,
as required.
\end{proof}

\begin{rem}
In the above proof of the implication (i)$\Rightarrow$(ii) in Theorem~\ref{t_ksm_plh}, we apply Theorem~\ref{t_aif}
to conclude that $\nu$ is a totally singular measure.
Below we obtain the required property of the measure $\nu$ without reference to Theorem~\ref{t_aif}.

Indeed, the Clark theorem in the unit disk guarantees that the measure $\mu_\za$ in \eqref{e_slice_mu}
has the following property: $|\mu_\za|(\Tbb\setminus E_{\za, \al})=0$,
where
\[
E_{\za, \al} = \{\lambda\in\Tbb: \lim_{r\to 1-} I(r\lambda\za) =\al\}.
\]
The measure $\mu_\za$ is defined for $\si_d$-a.e.\ $\za\in\spd$.
In fact, \eqref{e_slice_mu} and \eqref{e_nu} allow to identify $\mu_\za$ and the slice measure $\nu_\za$ of $\nu$.
Recall that the slice measure $\nu_\xi$ is defined for $\wsd$-a.e.\ $\xi\in\spd$.
So, $\nu_\xi = \mu_\xi$ for $\wsd$-a.e.\ $\xi\in\spd$.
Since $\nu\in\plh(\spd) \subset \dcm(\spd)$, Theorem~\ref{t_mu_mod} guarantees that
$|\nu|$ decomposes in terms of $|\nu_\xi|$.
Now, we apply Theorem~\ref{t_DM} to $|\nu|$ and $f=\chi_{E_{\al}}$, where
\[
E_{\al} = \{\za\in\spd: \lim_{r\to 1-} I(r\za) =\al\}.
\]
Since
$|\nu_\xi|(\spd\setminus E_{\al})=0$ for $\wsd$-a.e.\ $\xi\in\spd$,
we conclude that $|\nu|(\spd\setminus E_{\al})=0$.
By \cite[Theorem~9.3.2]{Ru80}, $E_\al$ is a totally null set;
hence, $\nu$ is a totally singular measure, as required.
\end{rem}


\section{Essential norms of composition operators}\label{s_clk_cmp}

In this section, we assume that $\ph: \bd\to\Dbb$, $d\ge 1$, is an arbitrary holomorphic function.
It is well-known that the composition operator $\cph: f\mapsto f\circ\ph$ sends $H^2(\Dbb)$
into $H^2(\bd)$. Indeed, let $f\in H^2(\Dbb)$. Then $|f|^2 \le h$ for an appropriate harmonic
function $h$ on $\Dbb$. So, $|f\circ \ph|^2 \le h\circ\ph$, hence $f\circ\ph \in H^2(\bd)$.

Two-sided estimates for the essential norm of the operator $\cph: H^2(\Dbb) \to H^2(\bd)$, $d\ge 1$,
were obtained by B.R.~Choe \cite{Ch92} in terms of the corresponding pull-back measure.
To prove Theorem~\ref{t_cmp_H2K_clk},
we use a more explicit approach proposed in \cite{ShJoel87} for $d=1$.

\subsection{Composition operators and Nevanlinna counting functions}\label{ss_cmp_neva}
Given an $f\in H^2(\Dbb)$, the
Littlewood--Paley identity states that
\begin{equation}\label{e_LP}
\|f\|^2_{H^2(\Dbb)} = |f(0)|^2 + 2\int_{\Dbb} |f^\prime(w)|^2 \log\frac{1}{|w|}\, d\aream(w),
\end{equation}
where $\aream$ denotes the normalized area measure on $\Dbb$.
To study the composition operator generated by a holomorphic self-map $\phi$ of the unit disk,
J.~H.~Shapiro \cite{ShJoel87} used an analogous formula for $f\circ\phi$
based on the Nevanlinna counting function $\neva_\phi$ defined as
\[
\neva_{\phi}(w) = \sum_{z\in \Dbb:\, \phi(z)=w} \log\frac{1}{|z|}, \quad w\in\Dbb\setminus\{\phi(0)\},
\]
where each pre-image is counted according to its multiplicity.

\begin{lemma}\label{l_Stanton}
Let $\ph: \bd\to\Dbb$, $d\ge 1$, be a holomorphic function.
Then
\begin{equation}\label{e_Stntn_d}
\|f\circ \ph\|^2_{H^2(\bd)} = |f(\ph(0))|^2 + 2\int_{\Dbb} |f^\prime(w)|^2
\left(\int_{\spd} \neva_{\ph_\za}(w)\,  d\si_d(\za) \right)\, d\aream(w).
\end{equation}
\end{lemma}
\begin{proof}
Let $\phi$ be a holomorphic self-map of $\Dbb$. Then
\begin{equation}\label{e_Stntn}
\|f\circ \phi\|^2_{H^2(\Dbb)} = |f(\phi(0))|^2 + 2\int_{\Dbb} |f^\prime(w)|^2 \neva_{\phi}(w)\, d\aream(w)
\end{equation}
by Stanton's formula; see \cite{ShJoel87}.

Applying \eqref{e_Stntn} to $\phi = \ph_\za$, $\za\in \spd$,
and integrating with respect to the normalized Lebesgue measure $\si_d$
on $\spd$, we obtain \eqref{e_Stntn_d}.
\end{proof}

Comparison of \eqref{e_Stntn_d} and \eqref{e_LP} suggests that $\cph: H^2(\Dbb)\to H^2(\bd)$ is a compact operator
if $\bneva(\ph)=0$, where
\[
  \bneva(\ph) = \limsup_{|w|\to 1-} \int_{\spd}\frac{\neva_{\ph_\za}(w)}{1-|w|}\,  d\si_d(\za).
\]
A formal verification of this implication is given in the proof of Lemma~\ref{l_ShJoel} below.

\subsection{Clark measures on the unit sphere and essential norms of composition operators}
Given a holomorphic function $\ph: \bd\to \Dbb$, $d\ge 1$,
let $\|\cph\|_{H^2, e}$ denote the essential norm of the composition
operator $\cph: H^2(\Dbb)\to H^2(\bd)$.

\begin{lemma}\label{l_ShJoel}
Let $\ph: \bd\to\Dbb$, $d\ge 1$, be a holomorphic function.
Then
\[
\bneva (\ph) \ge \|\cph\|_{H^2, e}^2.
\]
\end{lemma}
\begin{proof}
For $f\in H^2(\Dbb)$ and $0<r<1$, put $T_r f (z) =  f(rz)$, $z\in \Dbb$.
Clearly, $T_r$ is a compact operator on $H^2(\Dbb)$.

Let $f(z) = \sum_{k=0}^\infty a_k z^k$, $z\in \Dbb$, be in the unit ball of $H^2(\Dbb)$.
In particular, $|a_k|\le 1$.
By \eqref{e_Stntn_d},
\begin{align*}
  \|(\cph - \cph T_r) f\|_{H^2}^2
 & = |(f - f_r)(\ph(0))|^2 \\
 &\quad + 2 \int_{\Dbb} |(f-f_r)^\prime(w)|^2  \int_{\spd} \neva_{\ph_\za}(w)\, d\si_d(\za)\, d\aream(w).
\end{align*}
Firstly,
\[
|(f - f_r)(\ph(0))| \le \sum_{k=1}^\infty (1-r^k) |\ph(0)|^{k}\to 0 \quad\textrm{as}\ r\to 1-.
\]
Secondly, for $0<R<1$,
\begin{align*}
  &\int_{R\Dbb} |(f-f_r)^\prime(w)|^2  \int_{\spd} \neva_{\ph_\za}(w)\, d\si_d(\za)\, d\aream(w) \\
  &\quad\le \int_{\Dbb} \int_{\spd} \neva_{\ph_\za}(w)\, d\si_d(\za)\, d\aream(w) \left( \sum_{k=1}^\infty k(1-r^k) R^{k-1} \right)^2
\to 0 \quad\textrm{as}\ r\to 1-.
\end{align*}

Now, fix an $R\in (0,1)$ and put
\[
\bneva_R = \sup_{R\le |w| <1} \int_{\spd} \frac{\neva_{\ph_\za}(w)}{-\log |w|}\, d\si_d(\za).
\]
Let $\er>0$. The above estimates do not depend on $f$, $\|f\|_{H^2}\le 1$,
so, taking $r=r(\er, R)$ sufficiently close to $1$, we obtain
\begin{align*}
\|(\cph - \cph T_r) f\|_{H^2}^2
  &\le \er + 2\int_{\Dbb \setminus R\Dbb} |(f-f_r)^\prime(w)|^2 \bneva_R \log\frac{1}{|w|}\, d\aream(w) \\
  &\le \er + 2\bneva_R \int_{\Dbb} |(f-f_r)^\prime(w)|^2 \log\frac{1}{|w|}\, d\aream(w) \\
  &\le \er+ \bneva_R \|f-f_r\|_{H^2}^2 \\
  &\le \er+ \bneva_R
\end{align*}
by \eqref{e_LP}. Hence, $\|\cph\|_{H^2, e}^2 \le \bneva_R$, $0<R<1$. Since $-\log |w| \sim 1-|w|$ as $|w|\to 1-$,
the proof is finished.
\end{proof}

For $d=1$, the following lemma is proved in \cite{CM97}.

\begin{lemma}\label{l_ShJnn}
For $b\in \Dbb$, put
\[
f_b(w) = \frac{\sqrt{1-|b|^2}}{1 - \overline{b}w}, \quad w\in \Dbb.
\]
Let $\ph: \bd\to \Dbb$, $d\ge 1$, be a holomorphic function and let $\al\in\Tbb$. Then
\[
\|f_b \circ \ph\|^2_{H^2(\bd)} \to \|\clk_\al^s\| \quad \textrm{as}\ b\to\al\ \textrm{radially},
\]
where $\clk_\al^s$ denotes the singular part of the Clark measure $\clk_{\al} [\ph]$.
\end{lemma}
\begin{proof}
Let $b=r\al$, $0<r<1$, $\al\in\Tbb$. We have
\begin{align*}
\|f_b \circ \ph\|^2_{H^2(\bd)}
    &=\int_{\spd} \frac{1-r^2}{|\al- r\ph(\za)|^2}\, d\si_d(\za)\\
    &=\int_{\spd} \frac{1-|r\ph(\za)|^2}{|\al- r\ph(\za)|^2}\, d\si_d(\za)
     - r^2\int_{\spd} \frac{1-|\ph(\za)|^2}{|\al- r\ph(\za)|^2}\, d\si_d(\za)\\
    &= J_r - K_r.
\end{align*}
For any $w\in\Dbb$, the function
${r^2}{|1-rw|^{-2}}$
monotonically increases as $r\to 1-$. Thus,
\[
\lim_{r\to 1-} K_r = \int_{\spd} \frac{1-|\ph(\za)|^2}{|\al - \ph(\za)|^2}\, d\si_d(\za) = \|\clk_\al^a\|.
\]
Since
\[
J_r = \frac{1-|r\ph(0)|^2}{|\al - r\ph(0)|^2} \to \|\clk_\al\|\quad \textrm{as}\ r\to 1-,
\]
the proof is finished.
\end{proof}

Now, we are ready to prove the following extension of Theorem~\ref{t_cmp_H2K_clk}.

\begin{theorem}\label{t_cph_tri}
Let $\ph: \bd\to\Dbb$, $d\ge 1$, be a holomorphic function.
Then
\[
\|\cph\|_{H^2, e}^2 = \bclk (\ph) = \bneva (\ph),
\]
where
\[
  \bclk(\ph)  = \sup_{\al\in\Tbb} \|\clk_\al^s\|.
\]
\end{theorem}
\begin{proof}
By Lemma~\ref{l_ShJoel}, it suffices to obtain the following chain of inequalities:
\[
\|\cph\|_{H^2, e}^2
\overset{\textrm{(i)}}{\ge} \bclk \overset{\textrm{(ii)}}{\ge} \bneva.
\]

(i)
Let $f_b$, $b\in\Dbb$, be defined as in Lemma~\ref{l_ShJnn}.
Observe that $\|f_b\|_{H^2(\Dbb)} =1$ and $f_b \to 0$ weakly in $H^2(\Dbb)$ as $|b|\to 1-$.
So, given a compact operator $K: H^2(\Dbb) \to H^2(\bd)$, we have
$\|K f_b\|_{H^2(\bd)} \to 0$ as $|b|\to 1-$.
Hence,
\[
\|\cph -K\|^2_{H^2(\Dbb) \to H^2(\bd)}
\ge \limsup_{b\to \al} \|(\cph -K) f_b\|^2_{H^2(\bd)}
\ge \|\clk_\al^s\|
\]
by Lemma~\ref{l_ShJnn}.
Therefore, $\|\cph\|_{H^2, e}^2 \ge \bclk$, as required.

(ii) Let $\phi$ be a holomorphic self-map of $\Dbb$.
For $w\in \Dbb\setminus\{\phi(0)\}$ and $\al\in\Tbb$, put
\[
\nevva_\phi(w) = \int_{\Tbb} \log|\psi_w\circ\phi (\xi)|\, d\si_1(\xi)
+ \log\frac{1}{|\psi_w(\phi(0))|},
\]
where
\[
\psi_w(\lad) = \frac{w-\lad}{1- \overline{w}\lad}.
\]
As indicated in \cite[Sect~3.1]{NS04}, Jensen's formula and Fatou's lemma guarantee that
\[\neva_\phi (w)\le \nevva_\phi(w), \quad w\in\Dbb\setminus\{\phi(0)\}.
\]
Also, for $\lad \in \Dbb\cup\Tbb\setminus\{\al\}$, direct calculations show that
\[
-\lim_{w\to\al} \frac{\log |\psi_w (\lad)|}{1-|w|}
=\frac{1-|\lad|^2}{|\al- \lad|^2}.
\]
Therefore, applying the definition of $\nevva_{\ph_\za}$, $\za\in\spd$, and reverse Fatou's lemma, we obtain
\begin{align*}
  \limsup_{w\to\al}
  \int_{\spd} \frac{\neva_{\ph_\za}(w)}{1-|w|}\, d\si_d(\za)
  &\le \limsup_{w\to\al} \int_{\spd} \frac{\nevva_{\ph_\za}(w)}{1-|w|}\, d\si_d(\za)\\
  &\le \int_{\spd}
   \left(\frac{1-|\ph(0)|^2}{|\al- \ph(0)|^2} -
   \int_{\Tbb}\frac{1-|\ph_\za(\xi)|^2}{|\al- \ph_\za(\xi)|^2} \,d\si_1(\xi)\right)\, d\si_d(\za)  \\
  &=\frac{1-|\ph(0)|^2}{|\al- \ph(0)|^2} - \int_{\spd} \frac{1-|\ph(\za)|^2}{|\al- \ph(\za)|^2}\, d\si_d(\za) \\
  &=\|\clk_\al^s\|
\end{align*}
by \eqref{e_sclk_norm}.
So, we conclude that $\bneva[\ph]\le \bclk[\ph]$, as required. Hence, the proof of Theorem~\ref{t_cph_tri} is finished.
\end{proof}

\begin{corollary}\label{c_cph_tri}
Let $0<p<\infty$ and let $\ph: \bd\to\Dbb$, $d\ge 1$, be a holomorphic function.
Then the following properties are equivalent:
\begin{itemize}
  \item $\bneva (\ph) = 0$;
  \item $\cph: H^p(\Dbb)\to H^p(\bd)$ is a compact operator;
  \item all Clark measures $\clk_\al[\ph]$, $\al\in\Tbb$, are absolutely continuous.
\end{itemize}
\end{corollary}
\begin{proof}
If $p=2$, then Theorem~\ref{t_cph_tri} applies.
To finish the proof, it suffices to observe that
the operator $\cph: H^p(\Dbb)\to H^p(\bd)$ is compact for some $p>0$
if and only if it is compact for $p=2$; see, for example, \cite{McC85}.
\end{proof}


\end{document}